\documentclass{amsart}

\usepackage{amsthm,amsfonts,amsmath,amssymb}
\usepackage{lscape}
\usepackage{hyperref}
\usepackage{pdfpages}

\theoremstyle{plain}

\newtheorem{theorem}{Theorem}[section]
\newtheorem{lemma}[theorem]{Lemma}     
\newtheorem{corollary}[theorem]{Corollary}
\newtheorem{proposition}[theorem]{Proposition}

\theoremstyle{definition}

\theoremstyle{remark}

\newtheorem{remark}[theorem]{Remark}
\newtheorem{example}[theorem]{Example}

\newtheorem{remarkk}{Remark}
\newtheorem{claim}[remarkk]{Claim}
\newtheorem*{claim*}{Claim}

\DeclareMathOperator{\height}{ht}
\DeclareMathOperator{\supp}{supp}

\DeclareMathOperator{\ad}{ad}

 \newcommand{\calN}{\mathcal N}
\newcommand{\calO}{\mathcal O} 
\newcommand{\calX}{\mathcal X} 
\newcommand{\calS}{\mathcal S}

 \newcommand{\mN}{\mathbb N}
 
 \newcommand{\mZ}{\mathbb Z}

\newcommand{\goa}{\mathfrak a}
\newcommand{\gog}{\mathfrak g}

\newcommand{\gos}{\mathfrak s}

\newcommand{\gon}{\mathfrak n}
\newcommand{\gob}{\mathfrak b}

\newcommand{\got}{\mathfrak t}

\newcommand{\grd}{\delta}

\renewcommand{\setminus}      {\smallsetminus}
\renewcommand{\epsilon}      {\varepsilon}

\renewcommand{\geq}      {\geqslant}
\renewcommand{\leq}      {\leqslant}

\newsavebox{\kdwzero}
\newsavebox{\kdwone}
\newsavebox{\kdwtwo}
\savebox{\kdwzero}{\put(-240,240){0}}
\savebox{\kdwone}{\put(-240,240){1}}
\savebox{\kdwtwo}{\put(-240,240){2}}

\title{Fully commutative elements and spherical nilpotent orbits}

\date{\today}

\author{Jacopo Gandini}
\address{Dipartimento di Matematica, Universit\`a di Bologna, Piazza di Porta San Donato 5, 40126 Bologna, Italy}
\email{jacopo.gandini@unibo.it}

\begin{document}

\begin{abstract}
Let $\gog$ be a complex simple Lie algebra, with fixed Borel subalgebra $\gob \subset \gog$ and with Weyl group $W$. Expanding on previous work of Fan and Stembridge in the simply laced case, this note aims to study the fully commutative elements of $W$, and their connections with the spherical nilpotent orbits in $\gog$. If $\gog$ is not of type $G_2$, it is shown that an element $w \in W$ is fully commutative if and only if the subalgebra of $\gob$ determined by the inversions of $w$ lies in the closure of a spherical nilpotent orbit. A similar characterization is also given for the ad-nilpotent ideals of $\gob$, which are parametrized by suitable elements in the affine Weyl group of $\gog$ thanks to the work of Cellini and Papi.
\end{abstract}

\maketitle

\section{Introduction}

Let $G$ be a complex algebraic group, simple and adjoint with Lie algebra $\gog$. Let $T \subset G$ be a maximal torus with character lattice $\calX(T)$ and Lie algebra $\got$, and let $\Phi \subset \calX(T)$ be the associated root system, with Weyl group $W = N_G(T)/T$. Let also $B \subset G$ be a Borel subgroup containing $T$, $\gob \subset \gog$ the corresponding Lie algebra, whose nilradical is denoted by $\mathfrak n$, and $\Delta \subset \Phi$ the associated base. We denote the associated set of positive (resp. negative) roots by $\Phi^+$ (resp. $\Phi^-$).  If $\alpha \in \Phi$, we denote by $s_\alpha \in W$ the corresponding reflection, and by $\gog_\alpha \subset \gog$ the corresponding root space.
If $\alpha \in \Delta$, then the reflection $s_\alpha \in W$ is called simple. If $\alpha, \beta \in \Delta$, then we denote by $m(s_\alpha, s_\beta)$ the order of the product $s_\alpha s_\beta$.

Given $w \in W$, recall that
\begin{itemize}
	\item[$\bullet$] $w$ is \textit{commutative} if no reduced expression for $w$ contains a substring of the shape $s_\alpha s_\beta s_\alpha$, where $\alpha, \beta \in \Delta$ satisfy $||\alpha|| \leq ||\beta||$.\\
	\item[$\bullet$] $w$ is \textit{fully commutative} if no reduced expression for $w$ contains a substring of the shape $s_\alpha s_\beta s_\alpha s_\beta \ldots$ of length $m(s_\alpha, s_\beta)$, 	where $s_\alpha$ and $s_\beta$ are non-commuting simple reflections.
\end{itemize}

The sets of the commutative and of the fully commutative elements in $W$ will be respectively denoted by $W_c$ and  $W_{\!f\!c}$.

Commutative and fully commutative elements were introduced around the same time respectively by Fan \cite{Fan} and by Stembridge \cite{St} in the context of Coxeter groups. It follows from the definition that commutative elements are also fully commutative. 
In the simply laced case the two notions are equivalent: indeed in this case all the roots have the same length, and $m(s_\alpha, s_\beta) \leq 3$ for all $\alpha , \beta \in \Delta$. On the other hand, in the non-simply laced case fully commutative elements need not to be commutative. Consider for example the root system of type $B_2$, with simple roots $\alpha_1, \alpha_2$ such that $||\alpha_1|| > ||\alpha_2||$: then $m(\alpha_1,\alpha_2) = 4$, therefore $s_{\alpha_1} s_{\alpha_2} s_{\alpha_1}$ is fully commutative but not commutative.

Let $G$ act on $\gog$ via the adjoint action, denoted by $(g,x) \mapsto g.x$, and recall that $x \in \gog$ is called nilpotent if $\ad(x) \in \mathrm{End}(\gog)$ is nilpotent. Following a construction of Steinberg, every Weyl group element defines a nilpotent orbit in $\gog$ (that is, the $G$-orbit of a nilpotent element in $\gog$). In the simply-laced case, a connection between commutative elements in $W$ and nilpotent orbits in $\gog$ was pointed out by Fan and Stembridge \cite{FS} in terms of such construction: the (fully) commutative elements are precisely those elements which correspond to particular nilpotent orbits, called \textit{spherical}. The main goal of this note is to extend such a connection to the general case.

Let $w \in W$, and denote by
\[\Phi(w) = \{\alpha \in \Phi^+ \; | \; w(\alpha) \in \Phi^-\}\]
the corresponding set of inversions. We can attach to $w$ a $T$-stable subspace of $\gon$ by setting
\[
	\goa_w = \bigoplus_{\alpha \in \Phi(w)} \gog_\alpha = \mathfrak{n} \; \cap \; (w^{-1}w_0 \, \mathfrak{n})
\]
where $w_0 \in W$ denotes the longest element.

As pointed out by Fan \cite[Section 7]{Fan} (see also \cite{FS} for the case of arbitrary simply laced Coxeter groups), the commutativity of $w$ is equivalent to the fact that $\goa_w$ is an abelian subalgebra of $\gon$, namely that $\Phi(w)$ satisfies the following property:
\[
	\alpha, \beta \in \Phi(w) \quad \Longrightarrow \quad \alpha + \beta \not \in \Phi(w)
\]
More generally, a characterization of the full commutativity of $w$ in terms of the inversion set $\Phi(w)$ was given by Cellini and Papi \cite{CP} in the case of arbitrary Coxeter groups (see also Theorem \ref{teo:caratterizzazione_FC_CP}, where such characterization will be recalled).

Let $\calN = G.\gon$ be the subvariety of nilpotent elements in $\gog$. Since $\goa_w$ is irreducible and since the number of $G$-orbits in $\calN$ is finite, there exists a unique $G$-orbit $\calO_w \subset \calN$ which intersects $\goa_w$ in a dense open subset. Notice that, up to multiplication by $w_0$,  the map
\[
	\varphi : W \longrightarrow \calN/G \qquad \qquad w \longmapsto \calO_w
\]
is the Steinberg map $\mathrm{St} : W \longrightarrow \calN/G$  of \cite[Section 1.8]{Jo}. Namely, we have
\[\mathrm{St}(w) = \calO_{w_0 w^{-1}} = \calO_{w w_0} = \varphi(ww_0).\] In particular, up to multiplication by $w_0$, the fibers of $\varphi$ are the two-sided Steinberg cells of $W$. As a consequence of a result of Steinberg \cite[Theorem 3.5]{Stein}, it follows that $\mathrm{St}$ is a surjective map. Therefore $\varphi$ is a surjective map as well.

Recall that a nilpotent $G$-orbit is called \textit{spherical} if it contains an open $B$-orbit. The sphericality of the orbit of an element $x \in \calN$ is nicely characterized with the notion of \textit{height}, defined as
\[
	\height(x) = \max \{n \in \mN \; | \; \ad(x)^n \neq 0\}.
\]
Indeed, by a theorem of Panyushev \cite[Theorem 3.1]{Pa1} we have the equivalence
\[
	G.x \text{ is spherical} \Longleftrightarrow \height(x) \leq 3
\]

We will prove the following characterization, which is the first main result of the paper.

\begin{theorem}	\label{teo1-intro}
Let $w \in W$.
\begin{itemize}
	\item[i)] If $\Phi$ is simply laced or of type $G_2$, then $w \in W_c$ if and only if $\calO_w$ is spherical.
	\item[ii)] If $\Phi$ is doubly laced, then $w \in W_{\!f\!c}$ if and only if $\calO_w$ is spherical.
\end{itemize}
\end{theorem}

When $\Phi$ is simply laced, the previous theorem is due to Fan and Stembridge \cite{FS}. It seems that after that paper the general situation was never explored. The present note wants to fill such a lack.

As an application of the previous theorem, we get the following corollary (the case of type $G_2$ is treated on its own). 

\begin{corollary}	\label{cor1b}
The subset of the fully commutative elements $W_{\!f\!c} \subset W$ is a union of $w_0$-traslates of two-sided Steinberg cells.
\end{corollary}

Notice that in general the subset of the commutative elements $W_{\!c} \subset W$ is not a union of $w_0$-translates of two-sided Steinberg cells. This is indeed true in simply laced type because $W_{\!c} = W_{\!f\!c}$, and is true in type $G_2$ as well, but it fails in the doubly laced case.

Given a subspace $\goa \subset \gon$, we will say that $\goa$ is a \textit{spherical subspace} if every $G$-orbit intersecting $\goa$ is spherical, namely if $\height(x) \leq 3$ for all $x \in \goa$. Notice that the spherical locus
\[\calN_3 = \{x \in \calN \; | \; \height(x) \leq 3\}
\]
is a closed subvariety of $\calN$, which is indeed irreducible in all cases (see \cite[Section 6.1]{Pa2}). This implies that, if $w \in W$, then $\goa_w$ is a spherical subspace of $\gon$ if and only if $\calO_w$ is a spherical orbit, if and only if $\overline{G.\goa_w}$ is the closure of a spherical nilpotent orbit. Therefore Theorem \ref{teo1-intro} characterizes the sphericality of $\goa_w$ in terms of the commutativity and full commutativity of $w$.

We will also consider some kind of variation of Theorem \ref{teo1-intro}, allowing to characterize the sphericality of an ad-nilpotent ideal of $\gob$ in terms of the commutativity and full commutativity of a suitable element in the affine Weyl group $\widehat W$.

Following Cellini and Papi \cite{CP0}, we will say that an ideal $\goa \subset \gob$ is \textit{$\ad$-nilpotent} if it is contained in $\gon$. Notice that in this case $G.\goa$ is automatically closed, because $\goa$ is $B$-stable and the map $G \times^B \goa \rightarrow \gog$ defined by the $G$-action is proper.

Generalizing Peterson's classification \cite{Ko} of the abelian ideals of $\gob$, the ad-nilpotent ideals of $\gob$ were classified in \cite{CP0} by attaching to every such an ideal $\goa$ a suitable element $w_{\mathfrak a} \in \widehat W$ (see Section \ref{ssec:cp}, where the construction will be recalled).
On the other hand, the ad-\textit{nilpotent spherical ideals} (that is, the spherical subspaces of $\gon$ which are also ideals in $\gob$) were classified by Panyushev and R\"ohrle \cite{PR1}, \cite{PR2}. In these papers, they show that an abelian ideal of $\gob$ is necessarily spherical (and ad-nilpotent), and that if $\Phi$ is simply laced the converse is also true. Moreover, they classify the maximal ad-nilpotent spherical ideals for all complex simple Lie algebras.

The definition of commutative and fully commutative element given in the finite case go through in the affine case. We will prove the following theorem, which is the second main result of the paper.

\begin{theorem} \label{teo2-intro}
Let $\goa \subset \mathfrak{b}$ be an ad-nilpotent ideal and let $w_\goa \in \widehat W$ be the corresponding element.
\begin{itemize}
	\item[i)] If $\Phi$ is simply laced or of type $G_2$, then $\goa$ is spherical if and only if $w_\goa$ is commutative.
	\item[ii)] If $\Phi$ is doubly laced, then $\goa$ is spherical if and only if $w_\goa$ is fully commutative.
\end{itemize}
\end{theorem}

Again, in the simply laced case the previous theorem is well known. Indeed, by the mentioned results of Panyushev and R\"ohrle \cite{PR1} \cite{PR2}, in this case an ad-nilpotent ideal of $\gob$ is spherical if and only if it is abelian, and by Peterson's classification \cite{Ko} the abelian ideals $\goa \subset \gob$ are classified by the corresponding commutative elements $w_\goa \in \widehat W$. On the other hand, in the non-simply laced case the previous statement is missing in the literature.

We now discuss the organization of the paper. Sections \ref{sec:general} and \ref{sec:dimostrazione} deal with the case not of type $G_2$, whereas the latter is treated separately in Section \ref{sec:G2}.
More precisely, using the characterization of the fully commutative elements of Cellini and Papi \cite{CP}, in Section \ref{sec:general} we will reduce both Theorem \ref{teo1-intro} and Theorem \ref{teo2-intro} to another characterization of the sphericality of a subspace $\goa \subset \gon$, which is either defined by the inversions of a Weyl group element or is an ideal in $\gob$ (see Theorem \ref{teo:caratterizzazione-sspazi-sferici}). Section \ref{sec:dimostrazione} is devoted to the proof of Theorem \ref{teo:caratterizzazione-sspazi-sferici}.\\

\textit{Acknowledgments.} I want to thank Riccardo Biagioli, the conversations with whom raised my interest in the topic of this paper. As well, I want to thank John Stembridge for useful email exchanges, and the referee for his suggestions.

\section{Fully commutative elements and spherical nilpotent orbits} \label{sec:general}

Let $\Psi \subset \Phi^+$. Recall that $\Psi$ is called \textit{closed} if every root which is a sum of two elements in $\Psi$ is also in $\Psi$, whereas it is called \textit{convex} if every root which is a linear combination of two elements in $\Psi$ with positive coefficients is also in $\Psi$. If both $\Psi$ and $\Phi^+ \smallsetminus \Psi$ are closed (resp. convex), then $\Psi$ is called \textit{biclosed} (resp. \textit{biconvex}).

\begin{remark} \label{oss:biconvex}
If $\Psi \subset \Phi^+$, then the following are equivalent:
\begin{itemize}
	\item[i)] $\Psi = \Phi(w)$ for some $w \in W$;
	\item[ii)] $\Psi$ is biconvex;
	\item[iii)] $\Psi$ is biclosed.
\end{itemize}
The implication iii) $\Rightarrow$ i) is well known (see \cite[Exercise 16, Chap. VI, \S 1]{Bou}), whereas the other implications are clear. Notice that the equivalence i) $\Leftrightarrow$ ii) holds more generally for the root system of any finitely generated Coxeter group (see \cite[Section 8]{Pilk}), provided $\Psi$ is finite.
\end{remark}

Following Panyushev, we will say that $\Psi \subset \Phi^+$ is a \textit{combinatorial ideal} if it is closed under addition of arbitrary positive roots, namely if the following holds
\[
	\alpha\in \Psi, \, \beta \in \Phi^+ , \, \alpha + \beta \in \Phi^+ \Longrightarrow \alpha + \beta \in \Psi
\]

To any subset $\Psi \subset \Phi^+$ we attach a $T$-stable subspace of $\gon$, defined as follows
\[
	\goa_\Psi = \bigoplus_{\alpha \in \Psi} \mathfrak{g}_\alpha.
\]
Keeping the notation used in the introduction, if $\Psi = \Phi(w)$ we will also write $\goa_w$ instead of $\goa_{\Phi(w)}$.

\begin{remark}
Given $\Psi \subset \Phi^+$, it follows immediately from the definitions that
\begin{itemize}
	\item[i)] $\goa_\Psi$ is a subalgebra of $\gon$ if and only if $\Psi$ is a closed subset of $\Phi^+$, 
	\item[ii)] $\goa_\Psi$ is an ideal of $\gob$ if and only if $\Psi$ is a combinatorial ideal in $\Phi^+$.
\end{itemize}
\end{remark}

The formalism of biconvex subsets of roots and combinatorial ideals allows to treat with the same language the two problems investigated in the paper.
When $\Phi$ is not of type $G_2$, both Theorems \ref{teo1-intro} and \ref{teo2-intro} will be easy consequences of the following theorem, thanks to the characterization of Cellini and Papi \cite{CP} of the fully commutative elements of $\widehat W$ (see Theorem \ref{teo:caratterizzazione_FC_CP}). 

Denote by $(.,.)$ the $W$-invariant scalar product on $\langle \Phi \rangle_{\mathbb R}$ induced by the Killing form on $\gog$. For $\alpha, \beta \in \Phi$, set as usual $\langle \alpha,\beta\rangle = \tfrac{2(\alpha, \beta)}{(\beta, \beta)}$.

\begin{theorem}	\label{teo:caratterizzazione-sspazi-sferici}
Suppose that $\Phi$ is not of type $G_2$, and let $\Psi\subset \Phi^+$ be either a combinatorial ideal or a biconvex subset. Then $\goa_\Psi$ is a spherical subspace of $\gon$ if and only $\langle \alpha, \beta \rangle \geq 0$ for all $\alpha, \beta \in \Psi$.
\end{theorem}

As for Theorems \ref{teo1-intro} and \ref{teo2-intro}, when $\Phi$ is simply laced the previous theorem is well known. The case of biconvex subsets in simply laced type was considered by Fan and Stembridge \cite{FS}. The case of the combinatorial ideals in simply laced type can also be reduced to \cite{FS}, thanks to the mentioned result of Panyushev and R\"ohrle \cite{PR1} \cite{PR2} that an ideal of $\gob$ is spherical if and only if it is abelian, and that following Peterson's classification the abelian ideals of $\gob$ can be defined in terms of the inversions of a commutative element in $\widehat W$ (see \cite{Ko}).

When $\Psi \subset \Phi^+$ is a combinatorial ideal in non-simply laced type, the previous statement seems to be missing in the literature; however it is implicit in the classification of the maximal ad-nilpotent spherical ideals of $\gob$ of Panyushev and R\"ohrle \cite{PR2}.
It is indeed an easy remark to see that, if $\goa_\Psi$ is spherical, then $\langle \alpha, \beta \rangle \geq 0$ for all $\alpha, \beta \in \Psi$ (see Lemma \ref{lemma:2nonsferico}). On the other hand, if $\gog$ is doubly-laced, in order to classify the maximal ad-nilpotent spherical ideals of $\gob$, in \cite{PR2} the authors first classify the maximal combinatorial ideals $\Psi \subset \Phi^+$ such that $\langle \alpha, \beta \rangle \geq 0$ for all $\alpha, \beta \in \Psi$, and then show that the associated ideals are the maximal ad-nilpotent spherical ideals of $\gob$.


The proof of Theorem \ref{teo:caratterizzazione-sspazi-sferici} will be given in Section \ref{sec:dimostrazione}. It is basically an extension of the proof of Fan and Stembridge \cite{FS} from the simply laced case to the non-simply laced case. The lines of the proof are essentially the same, however technical issues make the proof in the general case definitely more involved (see Lemma \ref{lemma:4-ortogonali}, which constitutes the technical core of the proof). The case of the combinatorial ideals can be treated essentially in a uniform way together with the case of the biconvex subsets, only adding a few extra lines.

In the following subsections we explain how to deduce Theorems \ref{teo1-intro} and \ref{teo2-intro} from Theorem \ref{teo:caratterizzazione-sspazi-sferici} provided $\Phi$ is not of type $G_2$ (see Theorems \ref{teo1_ABCDEF} and \ref{teo2_ABCDEF}, respectively). The case of type $G_2$ will be discussed on its own in Section \ref{sec:G2}.

\subsection{Fully commutative elements in the affine Weyl group.}

Let $\widehat\Phi=\Phi\pm \mZ \grd$ be the affine 
root system defined by $\Phi$, where $\grd$ denotes the fundamental imaginary root (see \cite{kac}).
Let $\widehat \Phi^+ \subset \widehat \Phi$ be the positive subsystem containing $\Phi^+$ and all the roots of the shape $\alpha + n \delta$ with $\alpha \in \Phi$ and $n> 0$, and let $\widehat \Delta \supset \Delta$ be the corresponding base of $\widehat \Phi$. We also set $\widehat \Phi^- = - \widehat \Phi^+$.

Let $\widehat  W$ be the  Weyl group of $\widehat \Phi$. For $\alpha \in \widehat \Delta$ denote by $s_\alpha \in \widehat W$ the corresponding simple reflection. Extend the $W$-invariant scalar product $(.,.)$ on $\langle \Phi \rangle_{\mathbb R}$ to a $\widehat W$-invariant symmetric bilinear form on $\langle \widehat \Phi \rangle_{\mathbb R}$, with kernel generated by $\delta$. If $\alpha, \beta \in \widehat \Phi \setminus \mathbb Z \delta$ are real roots, set as usual $\langle \alpha, \beta \rangle = 2\tfrac{(\alpha,\beta)}{(\beta, \beta)}$.

Given $w \in \widehat W$, the set of inversions of $w$ is defined as
\[
	\widehat \Phi(w) = \{\alpha \in \widehat \Phi^+ \; | \; w(\alpha) \in \widehat \Phi^-\}.
\]
The notion of biconvex subset in $\widehat \Phi^+$ extends verbatim to the affine context, and still characterizes the inversion sets of elements in $\widehat W$ (see Remark \ref{oss:biconvex}).

The notions of commutative and fully commutative elements also extend verbatim to $\widehat W$. We now recall how these elements are characterized in terms of their sets of inversions.

The commutative elements were characterized by Fan and Stembridge \cite{FS}, \cite{St2}.

\begin{theorem}[{\cite[Theorem 2.4]{FS}, \cite[Theorem 5.3]{St2}}] \label{teo:FS}
Let $w \in \widehat W$, then $w$ is commutative if and only if $\alpha + \beta \not \in \widehat \Phi$ for all $\alpha, \beta \in \widehat \Phi(w)$.
%
%
If moreover $\Phi$ is simply laced, then $w$ is commutative if and only if $\langle \alpha, \beta \rangle \geq 0$ for all $\alpha, \beta \in \widehat \Phi(w)$.
\end{theorem}

The fully commutative elements were characterized by Cellini and Papi \cite{CP}. Recall that a subsystem $\Psi \subset \widehat\Phi$ is \textit{parabolic} if $\Psi = \langle \Psi \rangle_\mathbb R \cap \widehat \Phi$, namely if $\Psi$ is obtained by intersecting $\widehat \Phi$ with a linear subspace of $\langle \widehat \Phi \rangle _{\mathbb R}$. Every parabolic subsystem is generated by a subset of a suitable base of $\widehat \Phi$, and conversely. A \textit{parabolic positive subsystem} of $\widehat \Phi$ is a positive system of a parabolic subsystem.

\begin{theorem}[{\cite[Theorem 2]{CP}}] 	\label{teo:caratterizzazione_FC_CP}
Let $w \in \widehat W$, then $w$ is fully commutative if and only if $\widehat \Phi(w)$ does not contain any irreducible parabolic positive subsystem of rank 2.
\end{theorem}

Thus we obtain the following corollary, which generalizes the last statement of Theorem \ref{teo:FS}.

\begin{corollary}	\label{cor:caratterizzazione_FC}
Suppose that $\Phi$ is not of type $G_2$ and let $w \in \widehat W$. Then $w$ is fully commutative if and only if $\langle \alpha, \beta \rangle \geq 0$ for all $\alpha, \beta \in \widehat \Phi(w)$.
\end{corollary}

\begin{proof}
If $\Phi$ is of type $A_1$ the claim is easily shown directly, so we assume that $\Phi$ has rank at least 2. Then every parabolic subsystem $\Psi \subset \widehat \Phi$ of rank 2 is of finite type. If moreover $\Phi$ is not of type $G_2$, then the Dynkin diagram of $\widehat \Phi$ contains at most double bonds: thus the parabolic subsystems of $\widehat \Phi$ of rank 2 can only be of type $A_2$ or $B_2$. On the other hand, if $\Psi$ is of type $A_2$ or $B_2$, then two non-proportional elements $\alpha, \beta \in \Psi$ form a base of $\Psi$ if and only if $\langle \alpha, \beta \rangle < 0$. Therefore the claim follows from Theorem \ref{teo:caratterizzazione_FC_CP}.
\end{proof}

\begin{remark}
The statement of Corollary \ref{cor:caratterizzazione_FC} fails when $\Phi$ is of type $G_2$. Indeed in this case the fully commutative elements in $W$ are those of length at most $5$, whereas those which satisfy the condition of Corollary \ref{cor:caratterizzazione_FC} are those with length at most $4$.
\end{remark}

Going back to the case of the finite Weyl group, we get the validity of Theorem \ref{teo1-intro} when $\Phi$ is not of type $G_2$.

\begin{theorem}	\label{teo1_ABCDEF}
Suppose that $\Phi$ is not of type $G_2$, and let $w \in W$. Then $\goa_w$ is spherical if and only if $w$ is fully commutative.
\end{theorem}

\begin{proof}
Let $w \in W$, regarded inside $\widehat W$ via the inclusion $\Delta \subset \widehat \Delta$. Then $\widehat\Phi(w) = \Phi(w)$, thus by Corollary \ref{cor:caratterizzazione_FC} we have \[
	w \in W_{\!f\!c} \quad \Longleftrightarrow \quad \langle \alpha, \beta \rangle \geq 0 \quad \forall \alpha, \beta \in \Phi(w).
\]
On the other hand $\Phi(w) \subset \Phi^+$ is a biconvex subset, therefore the claim follows from Theorem \ref{teo:caratterizzazione-sspazi-sferici}.
\end{proof}

\subsection{Fully commutative elements in $\widehat W$ and spherical ideals of $\gob$} \label{ssec:cp}

In order to explain the case of the spherical ideals, we first need to recall how to attach to any ad-nilpotent ideal of $\gob$ an element in $\widehat W$. The construction is due to Cellini and Papi \cite{CP}, and is inspired by Peterson's classification of the abelian ideals of $\gob$ (see \cite{Ko}).


Given an ad-nilpotent ideal $\goa \subset \gob$ we denote by $\Psi_\goa \subset \Phi^+$ the corresponding combinatorial ideal, namely the set of positive roots defined by the equality $\goa = \goa_{\Psi_\goa}$.

Set $\Psi^{(1)}_\goa = \Psi_\goa$. For all integer $k > 1$, define inductively 
\[\Psi^{(k)}_\goa = (\Psi^{(k-1)}_\goa + \Psi_\goa) \cap \Phi^+.\]
Since $\goa$ is an ideal of $\gob$, we have $\Psi^{(k)}_\goa \subset \Psi_\goa$ for all $k >0$. Define now a set of positive roots $\widehat \Psi_\goa \subset \widehat \Phi^+$, setting
\[
	\widehat \Psi_\goa = \bigcup_{k \in \mathbb N} \{ k\delta - \alpha \; | \; \alpha \in \Psi^{(k)}_\goa \}.
\]

This is indeed a finite biconvex set of positive roots, which yields the following theorem.

\begin{theorem}[{\cite[Theorem 2.6]{CP}}]
If $\goa\subset \gob$ is an ad-nilpotent ideal, then there exists $w_\goa \in \widehat W$ such that $\widehat \Psi_\goa = \widehat \Phi(w_\goa)$.
\end{theorem}

If $\goa$ is abelian, then by definition we have 
$\Psi_\goa^{(k)} = \emptyset$ for all $k >1$.
If instead $\goa$ is spherical, then we have
$\Psi_\goa^{(k)} = \emptyset$ for all $k >2$ (see \cite[Proposition 4.1]{PR2}).

\begin{remark}	\label{oss:Psi_a-affine} 
Notice that $\goa$ is abelian if and only if $w_\goa$ is commutative. Since $\langle m \delta - \alpha, n \delta - \beta \rangle = \langle \alpha, \beta \rangle$, it holds moreover
\[
\langle \alpha, \beta \rangle \geq 0 \quad \forall \alpha, \beta \in \Psi_\goa
\quad \Longleftrightarrow \quad
\langle \alpha, \beta \rangle \geq 0 \quad \forall \alpha, \beta \in \widehat \Psi_\goa.
\]
\end{remark}

We can now deduce the validity of Theorem \ref{teo2-intro} when $\Phi$ is not of type $G_2$.

\begin{theorem} \label{teo2_ABCDEF}
Suppose that $\Phi$ is not of type $G_2$, and let $\goa \subset \gob$ be an ad-nilpotent ideal. Then $\goa$ is spherical if and only if $w_\goa$ is fully commutative.
\end{theorem}

\begin{proof}
By construction we have $\widehat \Phi(w_\goa) = \widehat \Psi_\goa$. Thus by Corollary \ref{cor:caratterizzazione_FC} together with Remark \ref{oss:Psi_a-affine} we get
\[
	w_\goa \in \widehat W_{\!f\!c}  \quad \Longleftrightarrow
	\quad \langle \alpha, \beta \rangle \geq 0 \quad \forall \alpha, \beta \in \Psi_\goa.
\]
On the other hand $\Psi_\goa \subset \Phi^+$ is a combinatorial ideal, therefore the claim follows from Theorem \ref{teo:caratterizzazione-sspazi-sferici}.
\end{proof}

\section{Proof of Theorem \ref{teo:caratterizzazione-sspazi-sferici}} \label{sec:dimostrazione}

Before going into the proof of Theorem \ref{teo:caratterizzazione-sspazi-sferici} we fix some further notation.  We denote by $\theta$ the highest root in $\Phi$. We also denote by $\theta_s$ the highest short root if $\Phi$ is non-simply laced, and we set $\theta_s  =\theta$ when $\Phi$ is simply laced. For all $\alpha \in \Phi$, we fix a non-zero element $e_\alpha \in \gog_\alpha$, and we denote by $U_\alpha \subset G$ the root subgroup with Lie algebra $\gog_\alpha$. In particular $U_\alpha$ is the image of a one parameter subgroup $u_\alpha : \mathbb C \rightarrow G$, whose action on $\gog$ is obtained by exponentiating the adjoint action of $e_\alpha$:
\[
	u_\alpha(\xi).x = x + \sum_{k >0} \frac{\xi^k}{k!} \,  \ad(e_\alpha)^k(x).
\]
Finally, if $x= \sum_{\alpha \in \Phi^+} c_\alpha e_\alpha$ is an element in $\gon$, we define the \textit{support} of $x$ as
\[\supp(x) = \{\alpha \in \Phi^+ \; | \; c_\alpha \neq 0\}.\]

From now on, we will assume throughout this section that $\Phi$ is an irreducible root system not of type $G_2$. The necessity of the condition of Theorem \ref{teo:caratterizzazione-sspazi-sferici} is well known (see e.g. \cite{FS}, \cite{PR2}). We recall a proof in the following lemma.

\begin{lemma}	\label{lemma:2nonsferico}
Let $\alpha, \beta\in \Phi^+$ be such that $\langle \alpha, \beta \rangle <0$. Then the $G$-orbit of $e_\alpha + e_\beta$ is not spherical.
\end{lemma} 

\begin{proof}
Denote $x = e_\alpha + e_\beta$ and let $\gos \subset \gog$ be the subalgebra generated by the root spaces $\gog_{\pm \alpha}$, $\gog_{\pm \beta}$. Then $\gos$ is a simple Lie algebra of type $A_2$ or $C_2$, and $x$ is a principal nilpotent element in $\gos$. A direct computation shows that $(\ad x_{|\mathfrak{s}})^4 \neq 0$. Thus $(\ad x)^4 \neq 0$, and $G. x$ cannot be spherical by Panyushev's characterization.
\end{proof}

\begin{corollary}
Let $\Psi \subset \Phi^+$ and suppose that $\goa_\Psi$ is a spherical subspace. Then $\langle \alpha, \beta \rangle \geq 0$ for all $\alpha, \beta \in \Psi$.
\end{corollary}

We now consider the other implication of Theorem \ref{teo:caratterizzazione-sspazi-sferici}.

If $x \in \goa_\Psi$ for some $\Psi \subset \Phi^+$, then $(\ad x)^4$ breaks into a sum of monomials of the shape
\[
\ad (x_{\gamma_1}) \ad (x_{\gamma_2}) \ad (x_{\gamma_3}) \ad (x_{\gamma_4})
\]
with $\gamma_i\in \Psi$ and $x_{\gamma_i} \in \gog_{\gamma_i}$ for every $i=1,2,3,4$. In particular, if one of such monomials is non-zero, there must exist $\gamma_1, \gamma_2, \gamma_3, \gamma_4 \in \Psi$ and $\alpha, \beta \in \Phi \cup\{0\}$ such that
\[\alpha = \beta + \gamma_1 + \gamma_2 +\gamma_3 + \gamma_4.\]
This leads us to consider similar configurations of roots.

\begin{remark} \label{oss3}
Let $\gamma_1, \gamma_2, \gamma_3, \gamma_4 \in \Phi^+$ be (not necessarily distinct) such that $\langle \gamma_i, \gamma_j \rangle \geq 0$ for all $i,j$, and let $\alpha, \beta \in \Phi \cup\{0\}$ be such that
\[
	\gamma_1 + \gamma_2 + \gamma_3 + \gamma_4 = \alpha - \beta.
\]
Notice that $\alpha \in \Phi \setminus \{\gamma_1, \gamma_2, \gamma_3, \gamma_4\}$. If indeed $\alpha = 0$, then we get a contradiction because
\[||\beta||^2 = \sum_{i=1}^4 ||\gamma_i||^2 + 2 \sum_{i < j} (\gamma_i, \gamma_j) \geq \sum_{i=1}^4 ||\gamma_i||^2 \geq 4 ||\theta_s||^2.\]
Similarly, if $\alpha \in \{\gamma_1, \gamma_2, \gamma_3, \gamma_4\}$, then $||\beta||^2 \geq 3 ||\theta_s||^2$, which is absurd because $\Phi$ is not of type $G_2$. Analogously, we also have $-\beta \in \Phi \setminus \{\gamma_1, \gamma_2, \gamma_3, \gamma_4\}$.

Notice also that $\{\gamma_1, \gamma_2, \gamma_3, \gamma_4\}$ has at least cardinality 2. Indeed every root string in $\Phi$ has length less than 4.
\end{remark}

We say that a subset $\Gamma \subset \Phi$ is \textit{orthogonal} if $\langle \gamma, \gamma' \rangle = 0$ for all $\gamma, \gamma' \in \Gamma$ with $\gamma \neq \gamma'$. The following lemma is the technical core of the proof.

\begin{lemma}	\label{lemma:4-ortogonali}
Suppose that $\Phi$ is not of type $G_2$. Let $\gamma_1, \gamma_2, \gamma_3, \gamma_4 \in \Phi^+$ be (not necessarily distinct) such that $\langle \gamma_i, \gamma_j \rangle \geq 0$ for all $i,j$, and denote $\Gamma = \{\gamma_1, \gamma_2, \gamma_3, \gamma_4 \}$. Suppose that $\Gamma$ is not orthogonal, and that
\begin{equation}	\label{eq:lemma}
	\gamma_1 + \gamma_2 + \gamma_3 + \gamma_4 = \alpha - \beta
\end{equation}
for some $\alpha, \beta \in \Phi$. Then $\Phi$ is doubly laced, and $\alpha = - \beta$ is a long root. Moreover there is at most one index $i_0$ such that $\gamma_{i_0}$ is a long root, and $\langle \gamma_{i_0}, \gamma_j \rangle = 0$ for all $j \neq i_0$.
\end{lemma}

\begin{proof}
We split the proof into several remarks and claims.

Notice that everything we will prove for $\alpha$ will be true for $-\beta$ as well, hence for $\beta$ with the obvious modifications. Indeed $\alpha-\beta = (-\beta) - (-\alpha)$, thus it is enough to replace $\alpha$ with $-\beta$, and $\beta$ with $-\alpha$.

\begin{remarkk}		\label{oss1}
\textit{For all $\gamma \in \Gamma$, it holds $\langle \alpha, \gamma \rangle \geq 0$ and $\langle \beta, \gamma \rangle \leq 0$.}
\end{remarkk}

We only consider the first statement, the other one is similar. By \eqref{eq:lemma} we have
\[\langle \alpha, \gamma \rangle \geq \langle \beta, \gamma \rangle + \langle \gamma, \gamma \rangle = \langle \beta, \gamma \rangle + 2.\] 
Thus we conclude because $\Phi$ is not of type $G_2$.
\hfill $\bigtriangleup$

\begin{remarkk} 	\label{oss2}
\textit{It holds $\sum_{i=1}^4 \langle \gamma_i, \alpha \rangle \leq 4$, with strict inequality if $\alpha$ is long and $\alpha \neq - \beta$}.
\end{remarkk}

By \eqref{eq:lemma} we have 
\[\sum_{i=1}^4 \langle \gamma_i, \alpha \rangle = \langle \alpha, \alpha \rangle - \langle \beta , \alpha \rangle \leq 2 - \langle \beta , \alpha \rangle.\]
Thus we conclude because $\Phi$ is not of type $G_2$.
\hfill $\bigtriangleup$

\begin{claim} \label{claim1}
\textit{$\Phi$ is doubly laced.}
\end{claim}

By \eqref{eq:lemma} we have
\begin{equation}	\label{eq:norma_rho}
	\frac{||\beta||^2}{||\theta_s||^2} = 
	\frac{||\alpha||^2}{||\theta_s||^2} 
	+ \sum_{i< j} \frac{||\gamma_j||^2}{||\theta_s||^2} \langle \gamma_i, \gamma_j \rangle  + \sum_{i=1}^4
 \frac{||\gamma_i||^2}{||\theta_s||^2} (1 - \langle \alpha, \gamma_i \rangle) 
\end{equation} 
Then the middle term on the right hand side is positive by the assumption on $\Gamma$. On the other hand $\alpha \not \in \Gamma$ by Remark \ref{oss3}, therefore if $\Phi$ is simply laced the last term of the equation is also non-negative, yielding $||\alpha|| < ||\beta||$, a contradiction.
\hfill $\bigtriangleup$

Since $\Phi$ is doubly laced, we also get the following formula for the last term appearing in \eqref{eq:norma_rho}, which will be useful later:
\begin{equation}	\label{eq:valori}
 \frac{||\gamma_i||^2}{||\theta_s||^2} (1 - \langle \alpha, \gamma_i \rangle) = \left\{
 \begin{array}{cl}
 2 & \text{if $\gamma_i$ is long and $\langle \alpha, \gamma_i \rangle = 0$} \\
1 & \text{if $\gamma_i$ is short and $\langle \alpha, \gamma_i \rangle = 0$} \\
0 & \text{if $\langle \alpha, \gamma_i \rangle = 1$} \\
-1 & \text{if $\langle \alpha, \gamma_i \rangle = 2$}
\end{array}  \right.
\end{equation}

\begin{claim} \label{claim1,5}
\textit{Both $\alpha$ and $\beta$ are long roots.}
\end{claim}

We show that $\alpha$ is long, the case of $\beta$ being similar.

Suppose that $\alpha$ is short. Since $\alpha \not \in \Gamma$ by Remark \ref{oss3}, it follows that $\langle \alpha, \gamma_i \rangle \leq 1$ for all $i$, thus all the terms appearing in \eqref{eq:norma_rho} are non-negative. On the other hand, by the assumption on $\Gamma$, the second term is positive. Therefore $||\beta|| > ||\alpha||$, namely $\beta$ is long and $||\beta||^2 = 2||\alpha||^2$ because $\Phi$ is not of type $G_2$. Moreover, we see that the third term on the right hand side of \eqref{eq:norma_rho} must be 0, namely it must be $\langle \alpha, \gamma_i \rangle = 1$ for all $i$. Thus $\langle \gamma_i, \alpha \rangle \geq 1$ for all $i$, hence $\langle \gamma_i, \alpha \rangle = 1$ for all $i$, thanks to Remark \ref{oss2}. In particular, $\Gamma$ contains only short roots.

For every $i= 1, \ldots, 4$, it follows from \eqref{eq:lemma} that
\[
	\langle \beta, \gamma_i \rangle = \langle \alpha, \gamma_i \rangle - \sum_{j=1}^4 \langle \gamma_j, \gamma_i \rangle = -1 - \sum_{j \neq i} \langle \gamma_j, \gamma_i \rangle < 0
\]
Since $\beta$ is long and all the $\gamma_i$ are short, it follows that $\langle \beta, \gamma_i \rangle = -2$ for all $i$. Thus for every $i$ there is a (unique) $j \neq i$ such that $\langle \gamma_i, \gamma_j \rangle \neq 0$. Summing both over $i$ and $j$, it follows that  
\[ \sum_{i < j} \langle \gamma_i, \gamma_j \rangle \geq 2\]
Thus by \eqref{eq:norma_rho} it follows that $||\beta||^2 \geq 3 ||\theta_s||^2$, which is absurd because $\Phi$ is not of type $G_2$.
\hfill $\bigtriangleup$

\begin{remarkk} \label{oss4}
\textit{	If $\gamma_i$ is long, then $\langle \gamma_i, \gamma_j \rangle = 0$ for all $j \neq i$.}
\end{remarkk}

Indeed by \eqref{eq:lemma} we have $0 \leq \sum_{j \neq i} \langle \gamma_j, \gamma_i \rangle = \langle \alpha, \gamma_i \rangle - \langle \beta, \gamma_i \rangle -2$. On the other hand $\alpha, -\beta \not \in  \Gamma$ by Remark \ref{oss3}, therefore the latter is non-positive by Remark \ref{oss1}.
\hfill $\bigtriangleup$

\begin{claim}
\textit{There is at most one index $i$ such that $\gamma_i$ is long.}
\end{claim}

If there are three long roots among the $\gamma_i$, then $\Gamma$ is necessarily orthogonal by Remark \ref{oss4}, a contradiction. Suppose that there are exactly two long roots, say $\gamma_1$ and $\gamma_2$. Since $\Gamma$ is not orthogonal, by Remark \ref{oss4} the short roots $\gamma_3$ and $\gamma_4$ must be different and non-orthogonal: hence it must be $\langle \gamma_3, \gamma_4 \rangle = 1$. On the other hand $\Phi$ is doubly laced by Claim \ref{claim1}, and  both $\alpha$ and $\beta$ are long roots by Claim \ref{claim1,5}. Thus $\langle \alpha, \gamma_4 \rangle$ and $\langle \beta, \gamma_4 \rangle$ are both even integers. But this yields a contradiction with \eqref{eq:lemma}, indeed
\[
	\langle \alpha, \gamma_4 \rangle - \langle \beta, \gamma_4 \rangle = \langle \gamma_3, \gamma_4 \rangle + 2 = 3
\]
thanks to Remark \ref{oss4}. \hfill $\bigtriangleup$

\begin{remarkk} \label{oss5}
\textit{There is at most one index $i$ such that $\langle \alpha, \gamma_i \rangle = 0$. In particular, $\sum_{i=1}^4 \langle \gamma_i, \alpha \rangle \geq 3$, with equality if and only if such an index exists.}
\end{remarkk}

Suppose that $\langle \alpha, \gamma_1 \rangle = \langle \alpha, \gamma_2 \rangle = 0$. Recall that $\alpha$ and $\beta$ are both long by Claim \ref{claim1,5}. Using the assumption on $\Gamma$, the equalities \eqref{eq:norma_rho} and \eqref{eq:valori} imply that
\[0> \sum_{i=1}^4
 \frac{||\gamma_i||^2}{||\theta_s||^2} (1 - \langle \alpha, \gamma_i \rangle)
 \geq 2 +  \sum_{i=3}^4
 \frac{||\gamma_i||^2}{||\theta_s||^2} (1 - \langle \alpha, \gamma_i \rangle) \geq 0,
\]
a contradiction. Since $\alpha$ is long, the last claim follows as well.
\hfill $\bigtriangleup$

\begin{claim} \label{claim2}
\textit{Suppose that $\sum_{i=1}^4 \langle \gamma_i, \alpha \rangle = 3$. Then we can enumerate $\gamma_1, \gamma_2, \gamma_3, \gamma_4$ in such a way that
\begin{align*}
	\langle \alpha, \gamma_1 \rangle = \langle \gamma_2, \gamma_1 \rangle = \langle \gamma_3, \gamma_1 \rangle =  \langle \gamma_4, \gamma_1 \rangle = 0, \\
	- \langle \beta, \gamma_1 \rangle = \langle \alpha, \gamma_2 \rangle = \langle \alpha, \gamma_3 \rangle = \langle \alpha, \gamma_4 \rangle = 2
\end{align*}
In particular, $\gamma_1, \gamma_2, \gamma_3, \gamma_4$ are all short roots. }
\end{claim}
 
By Remark \ref{oss5}, there is a unique index $i$ such that $\langle \gamma_i, \alpha \rangle = 0$. Assume $\langle \gamma_1 , \alpha \rangle = 0$. 

By Claim \ref{claim1,5} both $\alpha$ and $\beta$ are long. Therefore, by the assumption on $\Gamma$, equation \eqref{eq:norma_rho} yields
\[
	0  > \sum_{i=1}^4
 \frac{||\gamma_i||^2}{||\theta_s||^2} (1 - \langle \alpha, \gamma_i \rangle)
 \geq 1 +  \sum_{i=2}^4
 \frac{||\gamma_i||^2}{||\theta_s||^2} (1 - \langle \alpha, \gamma_i \rangle) 
 \]
By \eqref{eq:valori}, every term in the last sum is greater than or equal to $-1$. Thus at least two of its summands are negative, say $\langle \alpha, \gamma_2 \rangle = \langle \alpha, \gamma_3 \rangle = 2$. In particular, being $\alpha \not \in \Gamma$ by Remark \ref{oss3}, it follows that $\gamma_2$ and $\gamma_3$ are short roots.

%

On the other hand $\langle \alpha, \gamma_1 \rangle = 0$, thus by \eqref{eq:lemma} we get
\[
	\langle \beta, \gamma_1 \rangle = - \sum_{i\geq 1} \langle \gamma_i, \gamma_1 \rangle = -2
- \sum_{i> 1} \langle \gamma_i, \gamma_1 \rangle \leq -2\]
Therefore $\langle \beta, \gamma_1 \rangle = -2$ and $ \langle \gamma_i, \gamma_1 \rangle = 0$ for all $i > 1$. In particular, being $-\beta\not \in \Gamma$ by Remark \ref{oss3}, it follows that $\gamma_1$ is also a short root.

It only remains to show that $\langle \alpha, \gamma_4 \rangle = 2$. We already know $\langle \alpha, \gamma_4 \rangle > 0$, suppose $\langle \alpha, \gamma_4 \rangle = 1$. Since $\alpha$ is long, it follows that $\gamma_4$ is long as well. Thus by Remark \ref{oss4} we get $\langle \gamma_i, \gamma_4 \rangle = 0$ for all $i <4$. By equation \eqref{eq:norma_rho} it follows then
\[	
\langle \gamma_2, \gamma_3\rangle =
- \sum_{i=1}^4
 \frac{||\gamma_i||^2}{||\theta_s||^2} (1 - \langle \alpha, \gamma_i \rangle),
\]
and computing the right hand side we get $\langle \gamma_2, \gamma_3\rangle = 1$. Thus by \eqref{eq:lemma} we get
\[
	\langle \beta, \gamma_3 \rangle = \langle \alpha, \gamma_3 \rangle - \langle \gamma_2, \gamma_3 \rangle - \langle \gamma_3, \gamma_3 \rangle = -1
\]
On the other hand $\beta$ is long and $\gamma_3$ is short, thus $\langle \beta, \gamma_3 \rangle$ must be even. This yields a contradiction, and we conclude that $\langle \alpha, \gamma_4 \rangle = 2$.
\hfill $\bigtriangleup$

\textit{Conclusion of the proof.}
Recall that $\alpha$ is a long root by Claim \ref{claim1,5}. In order to prove that $\beta = - \alpha$, by Remark \ref{oss2} it is therefore enough to show that $\sum_{i=1}^4 \langle \gamma_i, \alpha \rangle = 4$.

Suppose that this is not the case. Then by Remark \ref{oss5} we have $\sum_{i=1}^4 \langle \gamma_i, \alpha \rangle = 3$. Thus by Claim \ref{claim2} equality \eqref{eq:norma_rho} reduces to
\begin{equation}	\label{eq:finale}
	\langle \gamma_2, \gamma_3 \rangle + \langle \gamma_2, \gamma_4 \rangle + \langle \gamma_3, \gamma_4 \rangle = 2
\end{equation}
By Claim \ref{claim2} and equation \eqref{eq:lemma}, evaluating $\langle \beta, \gamma_i \rangle$ for $i > 1$ yields the system
\[\left\{
\begin{array}{l}
 \langle \beta, \gamma_2 \rangle + \langle \gamma_3, \gamma_2 \rangle + \langle \gamma_4, \gamma_2 \rangle = 0 \\
	\langle \beta, \gamma_3 \rangle + \langle \gamma_2, \gamma_3 \rangle + \langle \gamma_4, \gamma_3 \rangle = 0 \\
		\langle \beta, \gamma_4 \rangle + \langle \gamma_2, \gamma_4 \rangle + \langle \gamma_3, \gamma_4 \rangle = 0
\end{array}
\right.\]

Notice that $\langle \gamma_i, \gamma_j \rangle = \langle \gamma_j , \gamma_i \rangle$ for all $i,j$: indeed the $\gamma_i$ all have the same length by Claim \ref{claim2}. Moreover $\beta$ is long and all the $\gamma_i$ are short, thus $\langle \beta, \gamma_i \rangle$ is either $0$ or $-2$ for all $i$.

If $\langle \gamma_i, \gamma_j \rangle = 1$ for some $i, j$, then we get $\langle \gamma_2, \gamma_3 \rangle = \langle \gamma_2, \gamma_4 \rangle = \langle \gamma_3, \gamma_4 \rangle = 1$,
contradicting \eqref{eq:finale}. Thus the pairings $\langle \gamma_i, \gamma_j \rangle $ can only take values 0 or 2.
On the other hand, since all the $\gamma_i$ have the same length, $\langle \gamma_i, \gamma_j \rangle = 2$ if and only if $\gamma_i = \gamma_j$. It follows that $\Gamma$ is an orthogonal set, a contradiction.
\end{proof}

\begin{example}
Examples of roots as in Lemma \ref{lemma:4-ortogonali} are the followings.
\begin{itemize}
	\item[i)] Suppose that $\Phi$ is of type $B_n$, represented in the usual way in terms of an orthonormal basis $\epsilon_1, \ldots, \epsilon_n$ of the vector space $\langle \Phi \rangle_\mathbb{R}$. Take as a set of positive roots
\[
	\Phi^+ = \{\epsilon_i \; | \; i=1, \ldots, n\} \cup \{\epsilon_i \pm \epsilon_j \; | \; 1 \leq i < j \leq n\}.
\]
Then we can take $\gamma_1 = \gamma_2 = \epsilon_i$ and $\gamma_3 = \gamma_4 = \epsilon_j$, whenever $i \neq j$.
	\item[ii)] Suppose that $\Phi$ is of type $C_n$, represented in the usual way in terms of an orthonormal basis $\epsilon_1, \ldots, \epsilon_n$ of the vector space $\langle \Phi \rangle_\mathbb{R}$. Take as a set of positive roots
\[
	\Phi^+ = \{\epsilon_i \pm \epsilon_j \; | \; 1 \leq i < j \leq n\} \cup \{2\epsilon_i \; | \; i=1, \ldots, n\}.
\]
Then we can take $\gamma_1 = \epsilon_i + \epsilon_j$, $\gamma_2 = \epsilon_i - \epsilon_j$, $\gamma_3 = \epsilon_i + \epsilon_k$, $\gamma_4 = \epsilon_i - \epsilon_k$, whenever $i \neq j$ and $i \neq k$.
	\item[iii)] When $\Phi$ is of type $F_4$, it is indeed possibile to have a configuration of roots in the previous lemma with a long root inside $\Gamma$. Let $\alpha_1, \alpha_2, \alpha_3, \alpha_4$ be the simple roots of $\Phi$, enumerated as in \cite{Bou}. Then we can take
\begin{itemize}
	\item[$\bullet$] $\gamma_1 = 	\alpha_1$,
	\item[$\bullet$] $\gamma_2 = 	\alpha_1 + 2 \alpha_2 + 2 \alpha_3 + \alpha_4$,
	\item[$\bullet$] $\gamma_3 = 	\alpha_1 + 2 \alpha_2 + 3 \alpha_3 + \alpha_4$,
	\item[$\bullet$] $\gamma_4 = 	\alpha_1 + 2 \alpha_2 + 3 \alpha_3 + 2\alpha_4$,
	\item[$\bullet$] $\alpha = 	2\alpha_1 + 3 \alpha_2 + 4 \alpha_3 + 2 \alpha_4$.
\end{itemize}
\end{itemize}
\end{example}

Before proving the second implication of Theorem \ref{teo:caratterizzazione-sspazi-sferici}, we consider the case of a subspace $\goa_\Gamma \subset \gon$, where $\Gamma$ either is an orthogonal set of roots, or satisfies the assumptions of Lemma \ref{lemma:4-ortogonali}.

We first consider the case of an orthogonal set $\Gamma \subset \Phi^+$, and show that, if it occurs inside a biconvex subset or in a combinatorial ideal of $\Phi^+$ which satisfies the condition of Theorem \ref{teo:caratterizzazione-sspazi-sferici}, then $\goa_\Gamma$ is a spherical subspace of $\gon$.

If $\Gamma \subset \Phi^+$, let $\Phi_\Gamma = \langle \Gamma \rangle_\mathbb R \cap \Phi$ be the parabolic subsystem of $\Phi$ generated by $\Gamma$.

\begin{remark} \label{oss:CFG}
Subsets of orthogonal roots $\Gamma \subset \Phi$ giving rise to non-spherical nilpotent orbits were studied in \cite{GMP}, \cite{GMP2}, and classified in \cite[Proposition 3.7]{CFG}. In particular, if $\Phi$ is not of type $G_2$ and $G.x_\Gamma$ is not a spherical orbit, we only have the following possibilities:
\begin{itemize}
	\item[($D_4$)] There are four pairwise distinct roots $\gamma_1, \gamma_2, \gamma_3, \gamma_4$ in $\Gamma$ such that 
	\[
	\tfrac{1}{2}(\gamma_1 + \gamma_2 + \gamma_3 + \gamma_4) \in \Phi.
	\]
If $\Gamma_0 = \{\gamma_1, \gamma_2, \gamma_3, \gamma_4\}$, then $\Phi_{\Gamma_0}$ is a root system of type $D_4$.\\
	\item[($B\!F_4$)] There are four pairwise distinct long roots $\gamma_1, \gamma_2, \gamma_3, \gamma_4$ in $\Gamma$ such that 
	\[
	\tfrac{1}{2}(\gamma_1 + \gamma_2) \in \Phi\qquad \text{and} \qquad \tfrac{1}{2}(\gamma_3 + \gamma_4) \in \Phi.
	\]
If $\Gamma_0 = \{\gamma_1, \gamma_2, \gamma_3, \gamma_4\}$, then $\Phi_{\Gamma_0}$ is a root system of type $B_4$ or $F_4$.\\
	\item[($B_3$)] There are two distinct long roots $\gamma_1, \gamma_2$ and a short root $\gamma_3$ in $\Gamma$ such that 
	\[
	\tfrac{1}{2}(\gamma_1 + \gamma_2 + 2\gamma_3) \in \Phi.
	\]
If $\Gamma_0 = \{\gamma_1, \gamma_2, \gamma_3\}$, then $\Phi_{\Gamma_0}$ is a root system of type $B_3$.
\end{itemize}
\end{remark}

\begin{lemma}	\label{lemma:ortogonali_sferici} 
Let $\Psi\subset \Phi^+$ be either a combinatorial ideal or a biconvex subset, and suppose that $\langle \alpha, \beta \rangle \geq 0$ for all $\alpha, \beta \in \Psi$. Let $\Gamma \subset \Psi$ be an orthogonal subset, then $\goa_\Gamma$ is a spherical subspace of $\gon$.
\end{lemma}

\begin{proof}
Arguing by contradiction, we show that if $\Gamma$ contains a subset $\Gamma_0$ as in cases ($D_4$), ($B\!F_4$), ($B_3$) of Remark \ref{oss:CFG}, then $\Psi$ cannot satisfy the assumptions of the lemma.

\begin{itemize}
\item[($D_4$)]
Notice that $\{\gamma_1, -\tfrac{1}{2}(\gamma_1 + \gamma_2 + \gamma_3 + \gamma_4 ), \gamma_2, \gamma_3\}$ is a base for $\Phi_{\Gamma_0}$. Since it is a root system of type $D_4$, it follows that
\[
\Phi_{\Gamma_0} = \{\pm \gamma_1,  \pm \gamma_2 , \pm \gamma_3, \pm \gamma_4\}	\cup
\{\tfrac{1}{2}(\pm \gamma_1 \pm \gamma_2  \pm \gamma_3 \pm \gamma_4)\}	
\]
Since $\gamma_1, \gamma_2, \gamma_3, \gamma_4$ are all positive, notice that $\Phi_{\Gamma_0}\cap \Phi^+$ must contain at least three elements of the shape
\[\tfrac{1}{2}(\pm \gamma_1 \pm \gamma_2  \pm \gamma_3 \pm \gamma_4)\]
having two coefficients equal to 1 and two coefficients equal to $-1$.
Therefore, up to reordering the $\gamma_i$'s, we can assume that
\[\beta := \tfrac{1}{2}(\gamma_1 + \gamma_2 - \gamma_3 - \gamma_4)
\qquad \qquad
\beta' := \tfrac{1}{2}(\gamma_1 - \gamma_2 - \gamma_3 + \gamma_4)\]
are both positive roots. Notice that $\beta, \beta' \in \Phi^+ \setminus \Psi$ by the assumption on $\Psi$: indeed $\langle \beta, \gamma_3 \rangle = \langle \beta', \gamma_3 \rangle = -1$. Similarly $\beta + \gamma_3 \in \Phi^+ \setminus \Psi$, because $\langle \beta + \gamma_3, \gamma_4 \rangle = -1$. Therefore $\Psi$ cannot be a combinatorial ideal. On the other hand $\gamma_1 = (\beta + \gamma_3) + \beta'$, therefore $\Psi$ cannot be a biconvex subset either.\\

\item[($B\!F_4$)]
Notice that the long roots of a root system of type $B_4$ or $F_4$ form a root system of type $D_4$. In particular, the set of the long roots $\Phi^\ell_{\Gamma_0}$ in $\Phi_{\Gamma_0}$ is a root system of type $D_4$.

Recall that two roots $\alpha, \beta$ are called \textit{strongly orthogonal} if $\alpha\pm \beta \not \in \Phi$. The following properties are easily shown:
\begin{itemize}
	\item[i)] In types $B_4$ and $F_4$, two orthogonal short roots are never strongly orthogonal. 
	\item[ii)] In type $F_4$, the half sum of two orthogonal long roots is always a root.
\end{itemize} 

As in case ($D_4$), it follows that
\[
\Phi^\ell_{\Gamma_0} = 
\{\pm \gamma_1,  \pm \gamma_2 , \pm \gamma_3, \pm \gamma_4\}	\cup
\{\tfrac{1}{2}(\pm \gamma_1 \pm \gamma_2  \pm \gamma_3 \pm \gamma_4)\}	
\]
Thus we can conclude as in the previous case.\\

\item[($B_3$)]
Notice that $\{\gamma_1, -\tfrac{1}{2}(\gamma_1 + \gamma_2 + 2\gamma_3 ), \gamma_3\}$ is a base for $\Phi_{\Gamma_0}$. Since it is a root system of type $B_3$, it follows that
\[
\Phi_{\Gamma_0} = \{\pm \gamma_1,  \pm \gamma_2 , \pm \gamma_3\}  \cup \{\tfrac{1}{2}(\pm \gamma_1 \pm \gamma_2)\} \cup \{\tfrac{1}{2}(\pm \gamma_1 \pm \gamma_2  \pm 2\gamma_3 )\}	 
\]
Up to switching $\gamma_1$ and $\gamma_2$, we can assume that
\[ \beta = \tfrac{1}{2}(\gamma_1 - \gamma_2)\]
is a positive root. Notice that $\beta \in \Phi^+ \setminus \Psi$, and that $\beta + \gamma_3 \in \Phi^+ \setminus \Psi$ as well: indeed $\langle \beta, \gamma_2 \rangle = \langle \beta + \gamma_3, \gamma_2 \rangle = -1$. Thus by the assumption $\Psi$ cannot be a combinatorial ideal. On the other hand, if $\Psi$ is a biconvex subset, consider the roots
\[
	\beta' := \tfrac{1}{2}(\gamma_1 + \gamma_2 - 2 \gamma_3) \qquad \qquad
	\beta'' = \tfrac{1}{2}(-\gamma_1 + \gamma_2 + 2 \gamma_3).
\]
Notice that either $\beta' \in \Phi^+$ or $\beta'' \in \Phi^+$, and that none of them can be in $\Psi$ since they assume negative values respectively against $\gamma_3$ and $\gamma_1$. On the other hand $\gamma_1 = (\beta + \gamma_3) + \beta'$ and $\gamma_3 = \beta + \beta''$, thus the biconvexity of $\Psi$ is contradicted in both cases.
\qedhere
\end{itemize}
\end{proof}

We now consider the non-orthogonal cases arising from Lemma \ref{lemma:4-ortogonali}. Let $\Gamma = \{\gamma_1, \gamma_2, \gamma_3, \gamma_4 \}$ be a set of positive roots satisfying the hypotheses of Lemma \ref{lemma:4-ortogonali}, and suppose that $\Gamma$ is contained in a combinatorial ideal or in a biconvex subset satisfying the condition of Theorem \ref{teo:caratterizzazione-sspazi-sferici}. In order to prove that $\goa_\Gamma$ is spherical, we first show that $\Gamma$ does not contain any long root.

\begin{lemma} \label{lemma:4-nonortF4}
Suppose that $\Phi$ is doubly laced, and let $\Psi \subset \Phi^+$ be either a combinatorial ideal or a biconvex subset such that such that $\langle \alpha, \beta \rangle \geq 0$ for all $\alpha, \beta \in \Psi$. Let $\gamma_1, \gamma_2, \gamma_3, \gamma_4 \in \Psi$ (not necessarily distinct) be such that $\tfrac{1}{2}(\gamma_1 + \gamma_2 + \gamma_3 + \gamma_4) \in \Phi$, and suppose that $\Gamma = \{\gamma_1, \gamma_2, \gamma_3, \gamma_4\}$ is not orthogonal. Then $\Gamma$ does not contain any long root.
\end{lemma}

\begin{proof}
Denote $\alpha = \tfrac{1}{2}(\gamma_1 + \gamma_2 + \gamma_3 + \gamma_4)$. By Lemma \ref{lemma:4-ortogonali}, we know that $\alpha$ is a long root not in $\Gamma$.  Therefore the equality
\[
	4 = \langle 2\alpha, \alpha \rangle = \sum_{i=1}^4 \langle \gamma_i, \alpha \rangle
\]
implies that $\langle \gamma_i, \alpha \rangle = 1$ for all $i$.

Suppose that $\gamma_1$ is a long root. Then Lemma \ref{lemma:4-ortogonali} shows that $\gamma_2, \gamma_3, \gamma_4$ are all short roots, and that $\langle \gamma_1, \gamma_i \rangle = 0$ for all $i > 1$. Since $\langle \gamma_i, \alpha \rangle = 1$, for every $i=2,3,4$ it follows that $\langle \alpha, \gamma_i \rangle = 2$.

Notice that $\langle \gamma_i, \gamma_j \rangle = 1$ whenever $i,j > 1$ and $i \neq j$. Indeed
\[
	4 = \langle 2 \alpha, \gamma_i \rangle = 2 + \sum_{j \in \{2,3,4\} \setminus \{i\}}
\langle \gamma_j, \gamma_i \rangle \]
for every $i=2,3,4$. In particular, up to reordering we may assume that $\gamma_i - \gamma_j \in \Phi^+$ whenever $2 \leq i < j \leq 4$.

Consider now the roots
\[
\beta_1 = -\gamma_1, \qquad \beta_2 = \alpha - 2 \gamma_2, \qquad \beta_3 = \gamma_2, \qquad \beta_4 = -\gamma_3
\]
It follows from the preceding remarks that $\big(\langle \beta_i, \beta_j \rangle\big)_{i,j}$ is a Cartan matrix of type $F_4$. Since two short (resp. long) roots in a root system of type $B$ (resp. $C$) are always orthogonal, it follows that $\Phi$ is itself of type $F_4$, and that $\{\beta_1, \beta_2, \beta_3, \beta_4\}$ is a base for $\Phi$. In particular, it follows that 
\begin{gather*}
\gamma_2 - \gamma_3 + \gamma_4 = \beta_1 + 2 \beta_2 + 4 \beta_3 + 2 \beta_4\\
\gamma_2 + \gamma_3 - \gamma_4 = - \beta_1 - 2 \beta_2 - 2 \beta_3 - 2 \beta_4
\end{gather*}
are both in $\Phi$, hence in $\Phi^+$ because $\gamma_2 - \gamma_3$ and $\gamma_3 - \gamma_4$ are both positive.

Denote $\Gamma_1 = \{\gamma_1, \gamma_2 - \gamma_3 + \gamma_4, \gamma_3\}$ and
$\Gamma_2 = \{\gamma_1, \gamma_2 + \gamma_3 - \gamma_4, \gamma_4\}$. Then $\Gamma_1$ and $\Gamma_2$ are two orthogonal sets of roots as in Case ($B_3$) of Remark \ref{oss:CFG}, thus by Lemma \ref{lemma:ortogonali_sferici} they cannot occur as subsets of $\Psi$. It follows that $\gamma_2 - \gamma_3 + \gamma_4 \in \Phi^+ \setminus \Psi$ and $\gamma_2 + \gamma_3 - \gamma_4 \in \Phi^+ \setminus \Psi$.

It follows that $\Psi$ cannot be a biconvex subset: indeed
\[
	2\gamma_2 = (\gamma_2 - \gamma_3 + \gamma_4) + (\gamma_2 + \gamma_3 - \gamma_4),
\]
therefore $\gamma_2 \not \in \Psi$, a contradiction.

Similarly, it follows that $\Psi$ cannot be a combinatorial ideal either: indeed
\[
	\gamma_2 + \gamma_3 - \gamma_4 = \gamma_2 + (\gamma_3 - \gamma_4),
\]
therefore $\gamma_2 + \gamma_3 - \gamma_4 \in \Psi$, a contradiction.
\end{proof}

Let now $\Psi \subset \Phi^+$ be a combinatorial ideal or a biconvex subset such that $\langle \gamma, \gamma' \rangle \geq 0$ for all $\gamma, \gamma' \in \Psi$. Suppose that $\gamma_1, \gamma_2, \gamma_3, \gamma_4 \in \Psi$ satisfy an equality of the shape
\[
	\gamma_1  + \gamma_2 + \gamma_3 + \gamma_4 = \alpha - \beta
\]
for some $\alpha, \beta \in \Phi \cup \{0\}$, and
denote $\Gamma = \{\gamma_1, \gamma_2, \gamma_3, \gamma_4\}$. Then by Lemma \ref{lemma:4-ortogonali} and Lemma \ref{lemma:4-nonortF4} either $\Gamma$ is orthogonal, or $\Phi$ is doubly laced and $\Gamma$ does not contain any long root. If $\Gamma$ is orthogonal, then $\goa_\Gamma$ is a spherical subspace of $\gon$ by Lemma \ref{lemma:ortogonali_sferici}. The following lemma shows that $\goa_\Gamma$ is spherical in the second situation as well.

\begin{lemma}	\label{lemma:sferici-non-ort}
Suppose that $\Phi$ is doubly laced. Let $\gamma_1, \gamma_2, \gamma_3, \gamma_4 \in \Phi^+$ be (not necessarily distinct) short roots such that $\langle \gamma_i, \gamma_j \rangle \geq 0$ for all $i,j$. Assume that $\Gamma = \{\gamma_1, \gamma_2, \gamma_3, \gamma_4\}$ is not orthogonal, and that $\tfrac{1}{2}(\gamma_1 + \gamma_2 + \gamma_3 + \gamma_4) \in \Phi$. Then, up to reordering, $\gamma_1 + \gamma_2 = \gamma_3 + \gamma_4$. Moreover, $\goa_{\Gamma}$ is a spherical subspace of $\gon$.
\end{lemma}

\begin{proof}
Denote $\alpha = \tfrac{1}{2}(\gamma_1 + \gamma_2 + \gamma_3 + \gamma_4)$. Notice that $\sum_{i=1}^4 \langle \gamma_i, \alpha \rangle = \langle 2\alpha, \alpha \rangle = 4$. By Lemma \ref{lemma:4-ortogonali} it follows that $\alpha$ is a long root, hence $\langle \gamma_i, \alpha \rangle = 1$ and $\langle \alpha, \gamma_i\rangle = 2$ for all $i=1,2,3,4$. Notice that
\[4 = \langle 2 \alpha, \gamma_i\rangle = 2 + \sum_{j=2}^4 \langle \gamma_j, \gamma_i \rangle.\]
Since $\langle \gamma_j, \gamma_i \rangle \in \{0,1\}$ for all $i\neq j$, we see that for every $i$ there is precisely one index $j$ such that $\langle \gamma_j, \gamma_i \rangle = 0$.

Up to reordering we may assume that $\langle \gamma_1, \gamma_2 \rangle = \langle \gamma_3, \gamma_4 \rangle = 0$. Notice that $\alpha - \gamma_1 \in \Phi$, and that $\langle \alpha - \gamma_1, \gamma_2 \rangle = 2$: thus $\alpha - \gamma_1 - \gamma_2 \in \Phi \cup \{0\}$. On the other hand $\alpha - \gamma_1 - \gamma_2$ is orthogonal with $\alpha$, with $\gamma_1$, and with $\gamma_2$. Thereofore it must be $\alpha - \gamma_1 -\gamma_2 = 0$, namely $\alpha =\gamma_1 +\gamma_2  = \gamma_3 + \gamma_4$.

We now show that $\height(x) = 2$ for all $x \in \goa_\Gamma$, which implies the claim by Panyushev's criterion. Since $\Phi$ is not of type $G_2$, this is true if $x \in \gog_\alpha$ for some $\alpha \in \Gamma$.

Since $\langle \gamma_1, \gamma_2 \rangle = \langle \gamma_3, \gamma_4 \rangle = 0$, it follows that $\gamma_2 - \gamma_1 = s_{\gamma_1}(\alpha)$ and $\gamma_4 - \gamma_3 = s_{\gamma_3}(\alpha)$ are both in $\Phi$. Thus up to reordering we may assume that $\gamma_2 - \gamma_1$ and $\gamma_4- \gamma_3$ are both in $\Phi^+$. Denote $\Gamma_1 = \{\gamma_1, \gamma_2\}$ and $\Gamma_2 = \{\gamma_3, \gamma_4\}$.

Notice that $\height(x) = 2$ for all $x \in \goa_{\Gamma_1}$. Let indeed $x_{\gamma_1} \in \gog_{\gamma_1}$ and $x_{\gamma_2} \in \gog_{\gamma_2}$ be non-zero, and consider the action of the one parameter subgroup $u_{\gamma_2 - \gamma_1}(\xi)$ on $x_{\gamma_1} + x_{\gamma_2}$. Being $2\gamma_2 - \gamma_1 \not \in \Phi$, we can find a value of the parameter $\xi_0$ such that
\[
	u_{\gamma_2 - \gamma_1}(\xi_0). (x_{\gamma_1} + x_{\gamma_2}) = x_{\gamma_1},
\]
hence $\height(x_{\gamma_1} + x_{\gamma_2}) = \height(x_{\gamma_1}) =2$. Similarly, we have $\height(x) = 2$ for all $x \in \goa_{\Gamma_2}$.

Let now $x \in \goa_\Gamma$, and suppose that $x \not \in \goa_{\Gamma_1}$ and $x \not \in \goa_{\Gamma_2}$. Since $\Gamma$ is not orthogonal, it must be $\Gamma_1 \neq \Gamma_2$. If $i \in \{1,2,3,4\}$, notice that $\gamma_i + \gamma_2 - \gamma_1 \in \Phi$ if and only if $i = 1$: indeed $\gamma_2 - \gamma_1$ is a long root, orthogonal both with $\gamma_3$ and $\gamma_4$. Similarly $\gamma_i + \gamma_4 - \gamma_3 \in \Phi$ if and only if $i = 3$. If $x \in \goa_\Gamma$, it follows that acting with $U_{\gamma_2-\gamma_1}$ and $U_{\gamma_4-\gamma_3}$ we may assume that $\supp(x) = \{\gamma_i, \gamma_j\}$ with $\gamma_i \in \Gamma_1$ and $\gamma_j \in \Gamma_2$.

Write $x = x_{\gamma_i} + x_{\gamma_j}$ with $x_{\gamma_i}  \in \gog_{\gamma_i}$ and $x_{\gamma_j} \in \gog_{\gamma_j}$. Then $\langle \gamma_i, \gamma_j \rangle = 1$, and $2\gamma_i -  \gamma_j \not \in \Phi$. Therefore, acting with the one parameter subgroup $u_{\gamma_i - \gamma_j}(\xi)$, for a particular value of the parameter $\xi_0$ we get
\[
	u_{\gamma_i - \gamma_j}(\xi_0). (x_{\gamma_i} + x_{\gamma_j}) = x_{\gamma_j}.
\]
Therefore $\height(x) = \height(x_{\gamma_j}) = 2$.
\end{proof}

We can now show that the condition in Theorem \ref{teo:caratterizzazione-sspazi-sferici} is also sufficient.

\begin{proposition}
Let $\Psi \subset \Phi^+$ be either a combinatorial ideal or a biconvex subset, and suppose that $\langle \alpha, \beta \rangle \geq 0$ for all $\alpha, \beta \in \Psi$. Then $\goa_\Psi$ is a spherical subspace.
\end{proposition}

\begin{proof}
Let $x \in \goa_\Psi$. Let $\mathcal S \subset \Psi$ be the support of $x$, and write $x = \sum_{\gamma \in \calS} x_\gamma$ with $x_\gamma \in \gog_\gamma$. Then
\begin{equation}	\label{adx4}
	(\ad x)^4 = \sum_{\gamma_1, \ldots, \gamma_4 \in \mathcal S} \ad (x_{\gamma_1}) \ad (x_{\gamma_2}) \ad (x_{\gamma_3}) \ad (x_{\gamma_4})
\end{equation}	
If a monomial
\[
\ad (x_{\gamma_1}) \ad (x_{\gamma_2}) \ad (x_{\gamma_3}) \ad (x_{\gamma_4})
\]
in the previous sum is non-zero, then there must exist $\alpha, \beta \in \Phi \cup\{0\}$ such that
\[\alpha = \beta + \gamma_1 + \gamma_2 +\gamma_3 + \gamma_4.\]
Notice that $\alpha, \beta$ are necessarily in $\Phi$, thanks to Remark \ref{oss3}.

If $\Gamma \subset \calS$ is a subset of cardinality $k \leq 4$, let $\mathrm{Full}_4(\Gamma)$ be the set of all surjective sequences $(\gamma_1, \gamma_2, \gamma_3, \gamma_4)$ of four elements in $\Gamma$, and set 
\[
	p_\Gamma = \sum_{(\gamma_1, \gamma_2, \gamma_3, \gamma_4)\in \mathrm{Full}_4(\Gamma)} \; \ad (x_{\gamma_1}) \ad (x_{\gamma_2}) \ad (x_{\gamma_3}) \ad (x_{\gamma_4})
\]
If $\alpha, \beta \in \Phi$, let $\mathrm{P}_4^+(\calS;\alpha, \beta)$ be the set of all $\Gamma \subset \calS$ of cardinality $k \leq 4$ such that
\[
	\gamma_1 + \gamma_2 + \gamma_3 + \gamma_4 = \alpha - \beta
\]
for some sequence $(\gamma_1, \gamma_2, \gamma_3, \gamma_4) \in \mathrm{Full}_4(\Gamma)$.

Denote $\mathrm P_4^+(\calS) = \bigcup_{\alpha, \beta \in \Phi} \mathrm P_4^+(\calS;\alpha,\beta)$. By definition, every non-zero monomial occurring in \eqref{adx4} appears as a summand of $p_\Gamma$ for a unique $\Gamma \in \mathrm P_4^+(\calS)$. Therefore
\[
	(\ad x)^4 = \sum_{\Gamma \in \mathrm{P}_4^+(\calS) } \; p_\Gamma
	\]
We show that $p_\Gamma = 0$ for all $\Gamma \in \mathrm P_4^+(\calS)$, which implies the statement thanks to Panyushev's characterization.

Let $\Gamma \in \mathrm P_4^+(\calS)$, and let $\alpha, \beta \in \Phi$ with $\Gamma \in \mathrm P_4^+(\calS;\alpha, \beta)$. Then Lemma \ref{lemma:4-ortogonali} together with Lemma \ref{lemma:4-nonortF4} show that either $\Gamma$ is an orthogonal set of roots,  or $\Phi$ is doubly laced, $\beta = -\alpha$ and $\Gamma$ does not contain any long root. Moreover,
it follows from Lemma \ref{lemma:ortogonali_sferici} and Lemma \ref{lemma:sferici-non-ort} that $\goa_\Gamma$ is a spherical subspace of $\gon$.

If $\Gamma' \subseteq \Gamma$, set $x_{\Gamma'} = \sum_{\gamma \in \Gamma'} x_\gamma$. Then 
\[
	0 = 	(\ad x_{\Gamma'})^4 = \sum_{\Gamma'' \subseteq \Gamma'} p_{\Gamma''}
\]
Arguing inductively on the cardinality of $\Gamma'$, we see that $p_{\Gamma'} = 0$ for all $\Gamma' \subseteq \Gamma$.
\end{proof}

\section{The case of type $G_2$.}	\label{sec:G2}

We now discuss the validity of Theorem \ref{teo1-intro} and Theorem \ref{teo2-intro} when $\Phi$ is of type $G_2$. In particular, keeping the notation already introduced, we will see that the following holds.

\begin{proposition}	\label{prop:G2}
Suppose that $\Phi$ is of type $G_2$.
\begin{itemize}
\item[i)] Let $w \in W$. Then $\goa_w$ is spherical if and only if $w$ is a commutative element.
\item[ii)] Let $\goa \subset \gob$ an ad-nilpotent ideal. Then $\goa$ is spherical if and only if $w_\goa \in \widehat W$ is a commutative element.
\end{itemize}
\end{proposition}

Concerning the second statement, it was already pointed out in \cite{PR2} that in type $G_2$ an ad-nilpotent ideal $\goa \subset \gob$ is spherical if and only if it is abelian, which means that $w_\goa$ is commutative.

Enumerate the base $\Delta = \{\alpha_1, \alpha_2\}$ so that $||\alpha_1|| < ||\alpha_2||$, and denote $s_1 = s_{\alpha_1}$ and $s_2 = s_{\alpha_2}$. We also denote by $\ell : W \rightarrow \mN$ the length function defined by $\Delta$, and by $\leq$ the Bruhat order on $W$. Let moreover $\preceq_L$ denote the left weak order on $W$ (see \cite[Section 3.1]{BB}), defined by
\[
	v \preceq_L w \Longleftrightarrow \Phi(v) \subseteq \Phi(w).
\]
For $\alpha, \beta \in \Phi$, let $N_{\alpha, \beta}$ be the structure constants of $\gog$, defined by the equalities
\[[e_\alpha, e_\beta] = N_{\alpha, \beta} \; e_{\alpha+\beta}
\qquad \qquad 
\text{ if } \alpha, \beta, \alpha+ \beta \in \Phi.
\]
Since the characteristic of the field is zero, we have $N_{\alpha, \beta} \neq 0$ whenever $\alpha+ \beta \in \Phi$.

Before proving the proposition, we first consider some nilpotent orbits in $\gog$ which are not spherical.

\begin{remark}	\label{oss:G2}
\begin{itemize}
\item[i)] Suppose that $\alpha, \beta \in \Phi$ are orthogonal roots, then $G.(e_\alpha + e_\beta)$ is not spherical. Indeed in this case $\alpha$ and $\beta$ are strongly orthogonal, and $\height(e_\alpha + e_\beta) = 4$ (see e.g. \cite{GMP2}). Thus $G.(e_\alpha + e_\beta)$ is not spherical by Panyushev's characterization.

\item[ii)] Suppose that $\alpha = \alpha_1$ and $\beta = 2\alpha_1+\alpha_2$, then $G.(e_\alpha + e_\beta)$ is not spherical. Denote indeed $\gamma = \alpha_1+\alpha_2$, then
\begin{align*}
	\qquad \qquad u_\gamma (\xi). & (e_\alpha + e_\beta) = e_\alpha + e_\beta+ \xi [e_\gamma, e_\alpha + e_\beta] + \tfrac{1}{2} \xi^2 [e_\gamma, [e_\gamma, e_\alpha]] = \\
%
	& = e_{\alpha_1} + (1+N_{\gamma,\alpha}\; \xi) \; e_{2\alpha_1+\alpha_2} + \tfrac{1}{2} N_{\gamma,\beta} \; \xi \; (2+N_{\gamma, \alpha} \; \xi)\; e_{3\alpha_1+2 \alpha_2}
\end{align*}
By choosing properly the value of the parameter $\xi$, it follows that
\[
	e_{\alpha_1} + e_{3\alpha_1+2\alpha_2} \in G.(e_\alpha + e_\beta)
\]
On the other hand $\alpha_1$ and $3\alpha_1+2\alpha_2$ are orthogonal, thus the claim follows by case i).

\item[iii)] Suppose that $\alpha = \alpha_1+\alpha_2$ and $\beta = 2\alpha_1+\alpha_2$, then $G.(e_\alpha + e_\beta)$ is not spherical. Indeed $s_{\alpha_2}(\alpha_1+\alpha_2) = \alpha_1$ and $s_{\alpha_2}(2\alpha_1+\alpha_2) = 2\alpha_1+\alpha_2$. Thus $e_{\alpha_1}+e_{2\alpha_1+\alpha_2} \in G.(e_\alpha+e_\beta)$, and we conclude thanks to the case ii).
\end{itemize}
\end{remark}

We can now analyze the variuos cases beyond Proposition \ref{prop:G2}.

\begin{proof}[Proof of Proposition \ref{prop:G2}]
Notice that $\gob$ contains a unique maximal abelian ideal $\goa_0$, whose corresponding set of roots is
\[
	\Psi_0 = \{2\alpha_1 + \alpha_2, \; 3\alpha_1 + \alpha_2, \; 3\alpha_1 + 2\alpha_2\}.
\]
i) We analyze the sphericality of the various subspaces $\goa_w$ with $w \in W$. 
Notice that $w$ is fully commutative if and only if $\ell(w) \leq 5$, whereas it is commutative if and only if $w \leq s_2 s_1 s_2$. In particular, if $w$ is not commutative, then either $s_1 s_2 s_1 \preceq_L w$ or $s_1 s_2 s_1 s_2 \preceq_L w$.

\begin{itemize}
	\item Suppose that $w$ is commutative, namely $w \leq s_2 s_1 s_2$. Then $\Phi(w)$ is conjugated to a subset of $\Psi_0$, thus $\goa_w$ is a spherical subspace.

	\item Suppose that $s_1 s_2 s_1 \preceq_L w$. Then $\Phi(w)$ contains both $\alpha_1$ and $2\alpha_1+\alpha_2$. By Remark \ref{oss:G2} ii), the element $e_{\alpha_1} + e_{2\alpha_1+\alpha_2}$ does not belong to a spherical orbit. Therefore $\goa_w$ is not a spherical subspace.

	\item Suppose that $s_1 s_2 s_1 s_2 \preceq_L w$. Then $\Phi(w)$ contains both $\alpha_2$ and $2\alpha_1+\alpha_2$, which is a pair of orthogonal roots. By Remark \ref{oss:G2} i), the element $e_{\alpha_2} + e_{2\alpha_1+\alpha_2}$ does not belong to a spherical orbit, thus $\goa_w$ is not a spherical subspace.
\end{itemize}

ii) By the definition of $w_\goa$, an ad-nilpotent ideal $\goa \subset \gob$ is abelian if and only if $w_\goa$ is commutative. Therefore it is enough to show that every $\ad$-nilpotent ideal $\goa \subset \gob$ which is not contained in $\goa_0$ is not spherical. On the other hand, if $\goa \not \subset \goa_0$, then we necessarily have $\alpha_1 + \alpha_2 \in \Psi_\goa$: thus $2\alpha_1 + \alpha_2 \in \Psi_\goa$ as well, and we conclude by Remark \ref{oss:G2} iii).
\end{proof}


\begin{thebibliography}{99}


\bibitem{BB}
A.~Bjorner and F.~Brenti,
\textit{Combinatorics of Coxeter groups},
Graduate Texts in Mathematics \textbf{231}, Springer-Verlag, Berlin, New York, 2005.


\bibitem{Bou}
N.~Bourbaki,
\textit{Groupes et Alg\`ebres de Lie. Chapitres 4--6,}
in: Actualit\'es Scientifiques et Industrielles, vol. 1337,
Hermann, Paris, 1968.

\bibitem{CP0}
P.~Cellini and P.~Papi,
\textit{ad-nilpotent ideals of a Borel subalgebra},
J. Algebra \textbf{225} (2000), 130--141.

\bibitem{CP}
P.~Cellini and P.~Papi,
\textit{A root system criterion for fully commutative and short-braid avoiding elements in affine Weyl groups}, J. Algebraic Combin. \textbf{11} (2000), 5--16.

\bibitem{CFG}
P.E. Chaput, L. Fresse and Th. Gobet,
\textit{Parametrization, structure and Bruhat order of certain spherical quotients}, Represent. theory \textbf{25} (2021), 935–-974.

\bibitem{Fan}
C.K.~Fan,
\textit{A Hecke algebra quotient and properties of commutative elements of a Weyl group}, PhD thesis, MIT, 1995.

\bibitem{FS}
C.K.~Fan and J.R.~Stembridge,
\textit{Nilpotent orbits and commutative elements}, J. Algebra \textbf{196} (1997), 490--498.



\bibitem{GMP} J.~Gandini, P.~M\"oseneder Frajria and P.~Papi,
\emph{Spherical nilpotent orbits and abelian subalgebras in isotropy representations},
J. Lond. Math. Soc. (2) \textbf{95} (2017), 323--352. 

\bibitem{GMP2} J.~Gandini, P.~M\"oseneder Frajria and P.~Papi,
\emph{Nilpotent orbits of height 2 and involutions in the affine Weyl group}, Indag. Math. \textbf{31} (2020), 568--594.

\bibitem{Jo} A.~Joseph, 
\textit{On the variety of a highest weight module}, 
J. Algebra \textbf{88}, (1984), 238--278.

\bibitem{kac}
V.~G. Kac, \emph{Infinite dimensional {L}ie algebras}, third ed., Cambridge University Press, Cambridge, 1990.

\bibitem{Ko}
B.~Kostant,
\emph{The set of abelian ideals of a Borel subalgebra, Cartan decompositions, and discrete series representations},
Int. Math. Res. Not. \textbf{5} (1998), 225--252.

\bibitem{Pa1} D.I.~Panyushev,
\textit{Complexity and nilpotent orbits},
Manuscripta Math. \textbf{83} (1994), 223--237.

\bibitem{Pa2}
D.I.~Panyushev,
\textit{On spherical nilpotent orbits and beyond},
Ann.\ Inst.\ Fourier (Grenoble) \textbf{49} (1999), 1453--1476. 


\bibitem{PR1} D.I.~Panyushev and G.~R\"ohrle,
\textit{Spherical orbits and abelian ideals},
\newblock Adv. Math. \textbf{159} (2001), 229--246.

\bibitem{PR2} D.I.~Panyushev and G.~R\"ohrle,
\textit{On spherical ideals of Borel subalgebras},
\newblock Arch. Math. \textbf{84} (2005), 225--232.


\bibitem{Pilk}
A.~Pilkington,
\textit{Convex Geometries on Root Systems},
Comm. in Algebra \textbf{34} (2006), 3183–-3202.

\bibitem{Stein}
R.~Steinberg,
\textit{On the desingularization of the unipotent variety},
\newblock Invent. Math. \textbf{36}  (1976), 209--224.

\bibitem{St}
J.R.~Stembridge,
\textit{On the fully commutative elements of a Coxeter group},
J. Algebraic Combin. \textbf{5} (1996), 353--385.

\bibitem{St2}
J.R.~Stembridge,
\textit{Minuscule elements of Weyl groups}, J. Algebra \textbf{235} (2001), 722--743.

\end{thebibliography}
\end{document}